\documentclass{amsart}
\usepackage{graphicx}

\newtheorem{theorem}{Theorem}
\newtheorem{lemma}[theorem]{Lemma}
\newtheorem{prop}[theorem]{Proposition}

\newtheorem{ques}[theorem]{Question}

\newcommand{\bibtitle}[1]{\emph{#1}}
\newcommand{\dfn}[1]{\textbf{#1}}

\newcommand{\ty}{\nabla\mathrm{Y}}
\newcommand{\yt}{\mathrm{Y}\nabla}
\newcommand{\Arf}{\mathrm{Arf}}

\newcommand{\badY}{\overline{\mathrm{Y}}}

\title{Many, many more intrinsically knotted graphs}
\author{Noam Goldberg}
\address{Department of Mathematics,
Occidental College,
Los Angeles, CA 90041}
\author{Thomas W.\ Mattman}
\address{Department of Mathematics and Statistics,
California State University, Chico,
Chico, CA 95929-0525}
\author{Ramin Naimi}
\address{Department of Mathematics,
Occidental College,
Los Angeles, CA 90041}

\thanks{The research was supported in part by the Undergraduate Research Center at Occidental College
and by DMS-0905300 at Occidental College}

\subjclass[2000]{Primary 05C10, Secondary 57M15, 57M25}

\keywords{spatial graphs, intrinsically knotted, triangle-Y move}
\date \today

\begin{document}

\begin{abstract}    

We list more than 200 new examples of minor minimal intrinsically knotted graphs and 
describe many more that are intrinsically knotted and likely minor minimal.

\end{abstract}

\maketitle


\section*{Introduction}

In the early 1980s Conway and Gordon~\cite{cg}
showed that every embedding of $K_7$, the complete graph on seven vertices,
in $S^3$ contains a nontrivial knot.
A graph with this property is said to be \dfn{intrinsically knotted} (IK).
The question ``Which graphs are IK?"
has remained open for the past 30 years.

A graph $H$ is a \dfn{minor} of another graph $G$
if $H$ can be obtained from a subgraph of $G$
by contracting zero or more edges.
A graph $G$ with a given property
is said to be \dfn{minor minimal} with respect to that property
if no proper minor of $G$ has the property.
It is easy to show that a graph is IK iff
it contains a minor that is minor minimal intrinsically knotted (MMIK).
Robertson and Seymour's Graph Minor Theorem~\cite{rs}
says that in any infinite set of graphs,
at least one is a minor of another.
It follows that for any property whatsoever,
there are only finitely many graphs that are minor minimal with respect to that property.
In particular, there are only finitely many MMIK graphs.
Furthermore, deciding whether one graph is a minor of another can be done algorithmically.
Hence, if we knew the finite set of all MMIK graphs,
we would be able to decide whether or not any given graph is IK.
However, obtaining this finite set,
or even putting an upper bound on its size,
has turned out to be very difficult.
In contrast, Robertson, Seymour, and Thomas \cite{rst} settled the corresponding question for \dfn{intrinsically linked} (IL) graphs
--- i.e., graphs for which every embedding in $S^3$ contains a nontrivial link --- in 1995:
there are exactly seven MMIL graphs;
they are obtained from $K_6$ by
$\ty$ and $\yt$ moves (we will define these shortly).

Prior to this work 41 MMIK graphs were known.
We have found 222 new MMIK graphs,
as well as many more IK graphs that are likely minor minimal.
In this paper we describe these 222 graphs.
For 101 of them we give a ``traditional'' proof that they are IK.
To prove that the remainder are also IK,
we rely on the computer program of \cite{mn}.
The program proves that a graph is IK
by showing every embedding of the graph
contains a $D_4$ minor with opposite cycles linked;
this is explained in greater detail in Section~\ref{sectionG28}.
We also prove that all 222 graphs are minor minimal.

First, some more
definitions and terminology.
A \dfn{spatial} graph is a graph embedded in $S^3$.
A spatial graph is said to be \dfn{knotted} (resp., \dfn{linked})
if it contains a nontrivial knot (resp., nontrivial link).
An abstract graph $G$ is \dfn{$n$-apex}
if one can remove $n$ vertices from $G$ to obtain a planar graph.
For an edge $e$ of $G$, $G - e$ denotes 
the graph obtained by removing $e$ from $G$ and $G/e$ 
the graph obtained by contracting $e$.

A \dfn{$\ty$ move} on an abstract graph
consists of removing the edges of a triangle (i.e., 3-cycle) $abc$ in the graph,
then adding a new vertex $v$ and connecting it
to each of the vertices $a$, $b$, and $c$,
as shown in Figure~\ref{triangleYmoveFigure}.
The reverse of this operation is called a
\dfn{$\yt$ move}.
Note that in a $\yt$ move,
the vertex $v$ cannot have degree greater than three.
(There is various terminology for this in the literature: 
$\nabla$ = triangle = Delta = $\Delta$; 
$\mathrm{Y}$ = wye = star;
move = exchange = transformation.)

\begin{figure}[ht]

 \centering
 \includegraphics[width=80mm]{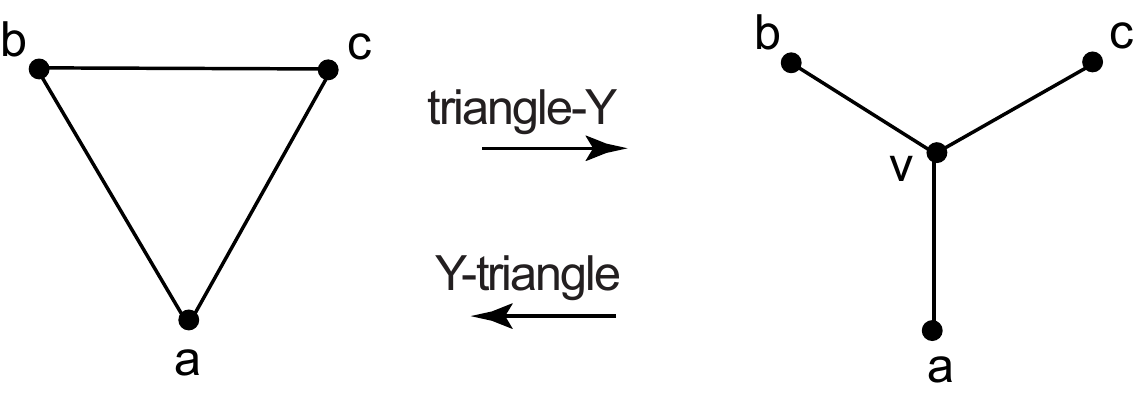}
 \caption{$\ty$ and $\yt$ moves.
 \label{triangleYmoveFigure}
 }
\end{figure}

If a graph $G'$ is obtained from a graph $G$
by exactly one $\ty$ move,
we say $G'$ is a \dfn{child} of $G$,
and $G$ is a \dfn{parent} of $G'$.
A graph that has no degree three vertices
can have no parents and we call such a graph
\dfn{parentless}; a triangle--free graph
has no children and is \dfn{childless}.
If $G'$ is obtained from $G$
by one or more $\ty$ moves,
we say $G'$ is a \dfn{descendant} of $G$,
and $G$ is an \dfn{ancestor} of $G'$.
If $G'$ is obtained from $G$
by \emph{zero} or more operations,
each of which is a $\ty$ or $\yt$ move,
we say $G$ and $G'$ are \dfn{cousins} of each other
(thus, being cousins is an equivalence relation).
The set of all cousins of $G$ is called the \dfn{$G$ family}.

Sachs \cite{sa} observed that every child of an IL graph is IL.
Essentially the same argument shows that every child of an IK graph is IK.
As a corollary of \cite{rst}, we also know that every parent of an IL graph is IL.
In contrast, it is shown in \cite{FN} that
a parent of an IK graph need not be IK.
In this paper we also use the following lemma:

\begin{lemma}
\label{lemYTMMIK}
\cite{bdlst, OT}
If an IK graph $G$ has a MMIK child,
then $G$ is MMIK.
\end{lemma}

In addition, we make frequent use of a lemma that is a consequence
of the observation (due, independently, to
\cite{BBFFHL} and \cite{OT}) that the join, $H * K_2$, of $H$ and $K_2$
is IK if and only if $H$ is nonplanar.

\begin{lemma} \label{lem2ap}
\cite{BBFFHL, OT}
If $G$ is $2$--apex, then $G$ is not IK.
\end{lemma}

\begin{table}[ht]
\begin{center}
\begin{tabular}{|l|c|c|c|c|} \hline
Family & Graphs & IK  & \multicolumn{2}{c|}{MMIK Graphs} \\
& (Total) & Graphs & Known & New \\ \hline \hline
$K_7$ & 20 & 14 & 14 & 0 \\ \hline
$K_{3,3,1,1}$ & 58 & 58 & 26 & 32 \\ \hline
$E_9+e$ & 110 & 110 & 0 & 33 \\ \hline
$G_{9,28}$ & 1609 & 1609 & 0 & 156 \\ \hline
$G_{14,25}$ & $>$ 600,000 & unknown & 0 & 1 \\ \hline
\end{tabular}
\caption{Families of the 222 new MMIK graphs.} \label{tblresults}
\end{center}
\end{table}

The graphs we study here fall into several families.
Below we give a quick overview of our results, which are summarized 
in Table~\ref{tblresults}; 
details are provided in the following sections, 
with one section devoted to each family.
Some of these families contain a large number of graphs;
we used a computer program to construct these families.
The $K_7$ family consists of 20 graphs,
14 of which were previously known to be MMIK.
We show the remaining 6 are not IK
(this was also shown, independently, in \cite{hnty}).
The $K_{3,3,1,1}$ family consists of 58 graphs,
26 of which were previously known to be MMIK.
We show the remaining 32 are also MMIK.
The $E_9 + e$ family consists of 110 graphs.
We show that all are IK
and exactly 33 of them are MMIK.
The $G_{9,28}$ family consists of 1609 graphs.
We show they are all IK
and at least 156 of them are MMIK.
For 101 of these 156 graphs, we prove the
graph is MMIK without making use of the computer program of \cite{mn}.
Sampling results obtained by computer suggest that
well over half of the graphs in this family are MMIK.
The $G_{14,25}$ family consists of over 600,000 graphs;
we don't know the exact number.
We only show that $G_{14,25}$ itself is MMIK.

Note that in each family all graphs have the same number of edges since
$\ty$ and $\yt$ moves do not change the number of edges in a graph.
However,  if two edges of a $\mathrm{Y}$ are part of a triangle,
then a $\yt$ move on that $\mathrm{Y}$ results in double edges (i.e., two edges with the same endpoints);
in this case we say that the initial graph has a $\badY$.
It turns out that there is no graph with a $\badY$ in
the families of each of the graphs
$K_7$, $K_{3,3,1,1}$,
$E_9+e$, $G_{9,28}$, and $G_{13,30}$. (This last graph is described below.)
The $G_{14,25}$ family, however, does contain graphs with a $\badY$.
Whenever our computer program that generates these families
encounters a $\badY$,
it does not perform a $\yt$ move on that $\badY$,
since the resulting graph, after deleting one of its double edges,
would have fewer edges than the initial graph.
(We prefer to consider graphs with different number of edges to be in distinct families.)

Note that a graph obtained by performing a $\yt$ move on a $\badY$
followed by deleting one of the resulting double edges
can also be obtained by
just contracting one of the edges in the $\badY$.
So it might be interesting to perform such $\yt$ moves
and study the resulting graphs;
they might lead to new MMIK graphs:
\begin{ques}
Find an example of a MMIK graph that results from contracting an
edge of a $\badY$ in the family of some other MMIK graph.
\end{ques}
In particular, this would be a way to move from the family of one MMIK 
graph to that of another. We will not pursue this further here as our examples
of a  $\badY$ are in the $G_{14,25}$ family, which is already huge
even without considering additional graphs constructed in this way.

Although verifying that a given graph is MMIK can be laborious,
using our computer program to generate new candidates for MMIK graphs
turned out to be relatively quick.
Considering the ease with which we found families of new MMIK graphs,
we expect there are many more such families.
Since we know of no upper bound on the number of MMIK graphs,
it seems that until there is more progress in the theory,
it may not be worthwhile to continue the search for more MMIK graphs.

Instead, we propose a couple of questions regarding the
size of a MMIK family. The $G_{14,25}$ example shows that these families can become quite large.
However, the question of the smallest family remains open. In particular, we can ask:

\begin{ques}
Is there a MMIK graph that is its own family?
\end{ques}

Such a graph would be both childless and parentless. One way to approach this might be to
investigate the following.

\begin{ques}
Given an arbitrary graph, is there an efficient way of finding,
or at least estimating, how many cousins it has?
\end{ques}

For example, the MMIK graph described by Foisy in \cite{f2},
which has 13 vertices and 30 edges ($G_{13,30}$),
has more edges than any of the graphs mentioned above.
So we were surprised to learn that
its family consists of only seven graphs,
making it the smallest family known to us 
that contains a MMIK graph.

Finally, we remark that our study includes a description of
four new MMIK graphs on nine vertices:
$E_9+e$, $G_{9,28}$, and Cousins 12 and 41 of the $K_{3,3,1,1}$ family.
A computer search~\cite{M} suggests that these, along with the known
(i.e., as in \cite{KS}) MMIK graphs in the
$K_7$ and $K_{3,3,1,1}$  families, form a complete list of MMIK graphs on nine or fewer vertices. 
In particular, we expect that the families described in this paper 
include all MMIK graphs with at most nine vertices.

\section{The $K_7$ Family}

\begin{figure}[ht]
\begin{center}
\includegraphics[scale=0.7]{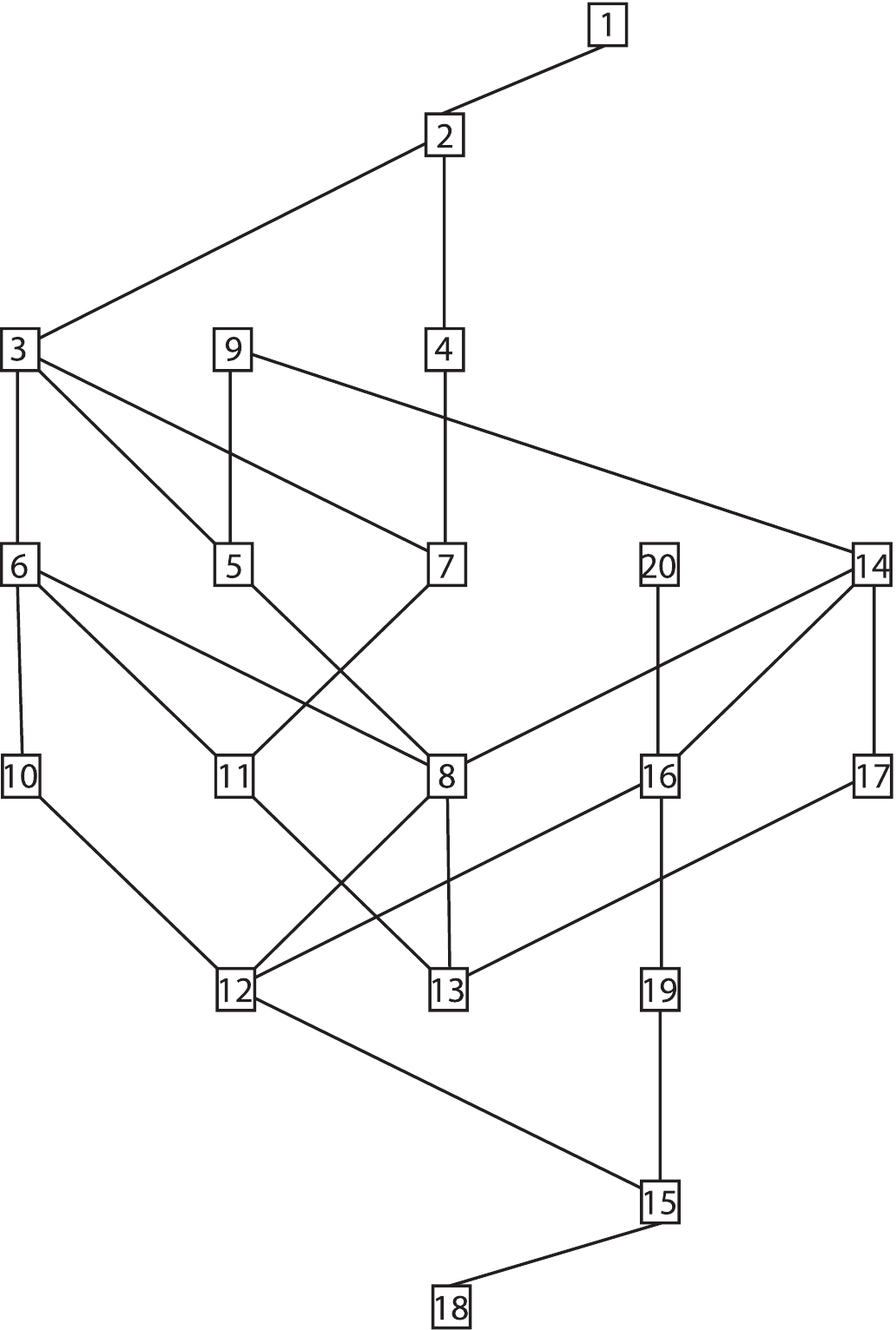}
\caption{The $K_7$ family.}\label{figK7F}
\end{center}
\end{figure}

\begin{figure}[ht]
\begin{center}
\includegraphics[scale=0.35]{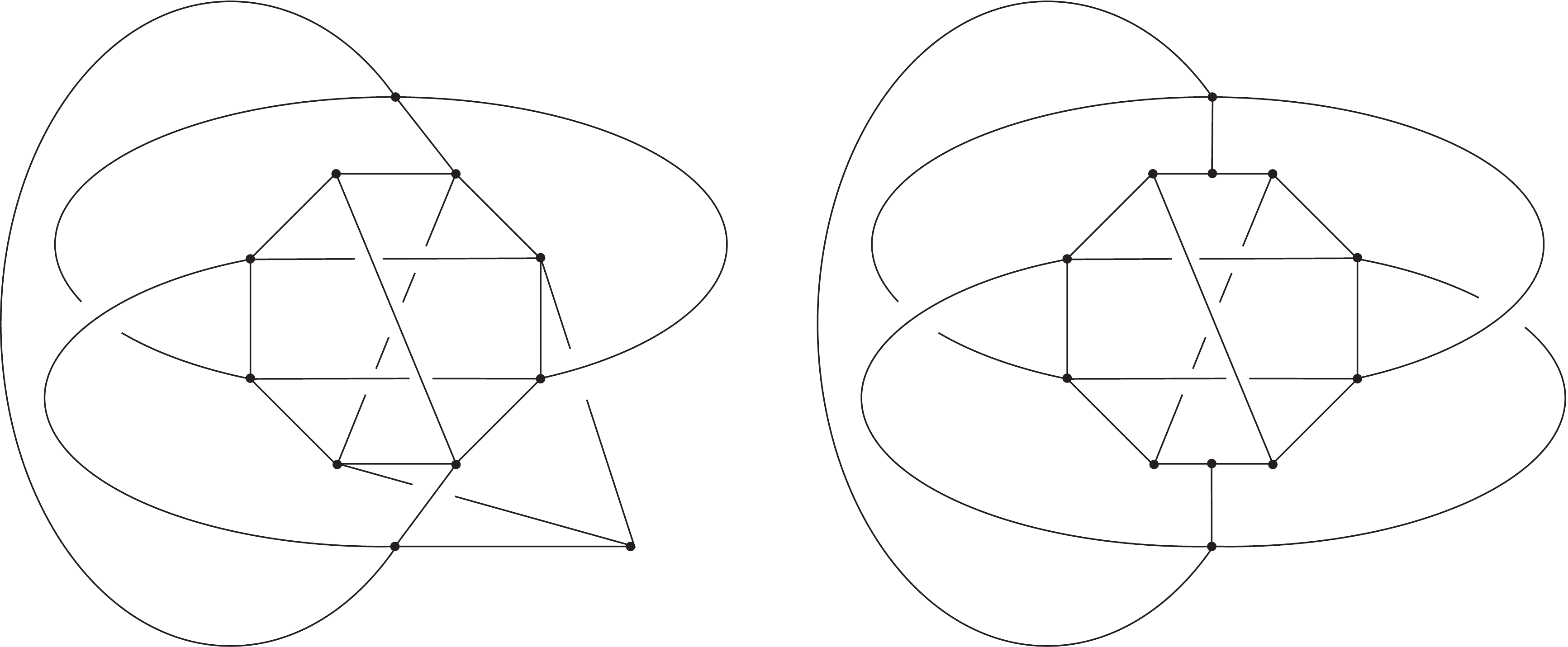}
\caption{Unknotted embeddings of Cousin 17 (left) and Cousin 19 (right) of $K_7$.}
\label{figC179}
\end{center}
\end{figure}

Figure~\ref{figK7F} shows the family of 20 graphs derived from $K_7$ by
$\yt$ and $\ty$ moves. 
An edge list for each of these 20 graphs can be found in the Appendix~\cite{GMN}.
Graphs at the same horizontal level have the same number of vertices, beginning with $K_7$ (Cousin~1) at the top and concluding with a 14--vertex graph, $C_{14}$ (Cousin~18), at bottom. Edges join parent to child.
The numbering of the cousins is somewhat arbitrary:
it reflects the order in which these graphs were constructed
via $\ty$ and $\yt$ moves by our algorithm.
Note that Cousin~9 is labeled $E_9$ in \cite{Ma}, 
and Cousins~16 and 20 are labeled $G_6$ and $G_7$ in \cite{FN}.

Kohara and Suzuki~\cite{KS} earlier described $K_7$ and its 13 descendants.
None of the six remaining cousins, 9, 14, 16, 17, 19, and 20 are IK. 
This follows as Cousins~17 and 19 have unknotted embeddings, as shown in Figure~\ref{figC179},
and Cousins~9, 14, 16, 20 are ancestors of Cousins~17 and 19.
(The unknotted embeddings of Figure~\ref{figC179} were derived
from the unknotted embedding of Cousin~20 that appears as
Figure~2 in \cite{FN}.)
Thus, the $K_7$ family yields no new examples of MMIK graphs.
This has also been shown, independently,
by Hanaki, Nikkuni, Taniyama, and Yamazaki \cite{hnty}.

We remark that $E_9$ (Cousin~9), a graph on nine vertices
and 21 edges, is 
the smallest graph that is not IK but has an IK child.
Indeed, it follows from \cite{Ma} that descendants of a 
non-IK graph on fewer edges or fewer vertices would be $2$-apex and,
therefore, not IK by Lemma~\ref{lem2ap}. Descendants of a graph
on 20 or fewer edges also have 20 or fewer edges and, so, are $2$-apex.
As for graphs on eight vertices, the non-IK examples with 21 or more edges
are all subgraphs of two graphs, $G_1$ and $G_2$, on eight vertices and
25 edges (see \cite[Figure 7]{Ma}). As all descendants of these 
two graphs are $2$-apex, the same is true of descendants of any 
subgraphs of $G_1$ or $G_2$.

It turns out that by adding one edge to $E_9$ one can obtain a MMIK graph;
we call this graph $E_9 + e$ and describe its family in Section~\ref {sectionE9eFamily}.

\section{The $K_{3,3,1,1}$ Family}

\begin{figure}[hb]
\begin{center}
\includegraphics[scale=0.77]{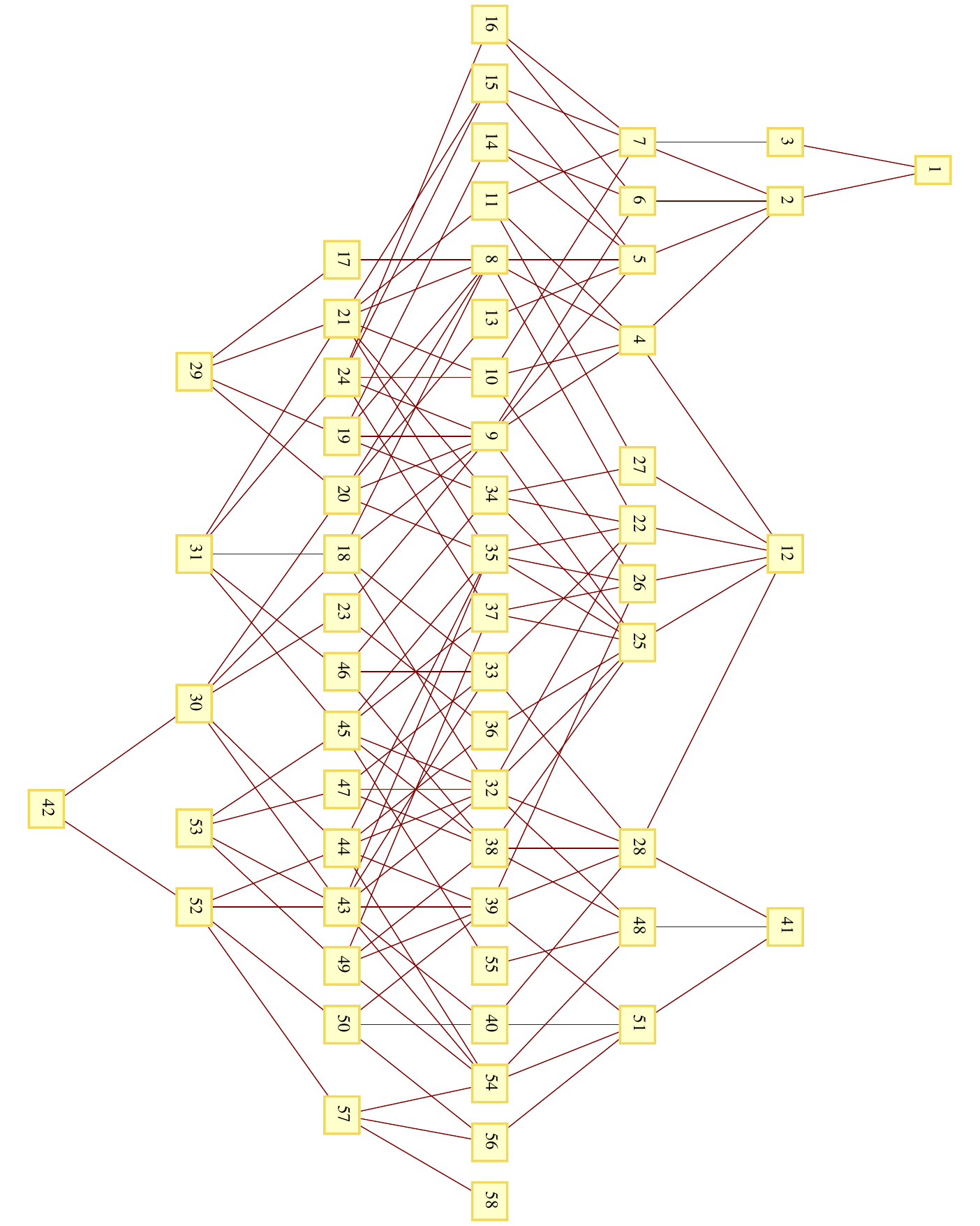}
\caption{The $K_{3,3,1,1}$ family.}\label{figK3311F}
\end{center}
\end{figure}

Figure~\ref{figK3311F} (produced using Mathematica) shows the 58 graphs derived from $K_{3,3,1,1}$ by $\ty$ and $\yt$ moves.
Edge lists for these graphs can be found in the Appendix~\cite{GMN}.
The graphs range from the 8 vertex graph $K_{3,3,1,1}$ (Cousin 1) through the 14 vertex graph Cousin 42 (called $R_1$ in \cite {OT}). 
Kohara and Suzuki~\cite{KS} described the graph $K_{3,3,1,1}$ and its 25 descendants. 
These 26 graphs were already known to be MMIK~\cite{F,KS}. As we will now show, the remaining 32 graphs in the family
are also MMIK.

\begin{prop} The 58 graphs in the $K_{3,3,1,1}$ family are all MMIK.
\end{prop}

\begin{proof}
We first observe that all graphs in the family are IK. For this, it suffices to show that the four parentless cousins, 1, 12, 41, and 58, are
intrinsically knotted. Foisy~\cite{F} proved this for Cousin 1, $K_{3,3,1,1}$.
We handle the remaining three graphs by using the computer program described in \cite{mn}
to verify that in every embedding of the graph
there is a $D_4$ minor that contains a knotted Hamiltonian cycle.

Having established that all graphs in the family are IK, by Lemma~\ref{lemYTMMIK}, we can conclude that they are all MMIK once we've shown
this for the four childless cousins, 29, 31, 42, and 53. We do know that 
descendants of $K_{3,3,1,1}$ are MMIK. This combines work  
of Kohara and Suzuki~\cite{KS} (who argued that, if $K_{3,3,1,1}$ is MMIK,
then all of its descendants are too)
and Foisy~\cite{F} (who proved that $K_{3,3,1,1}$ is MMIK). As cousins 29, 31, and 42 
have $K_{3,3,1,1}$ as an ancestor, the following lemma completes the argument.
\end{proof}

\begin{figure}[ht]
\begin{center}
\includegraphics[scale=0.6]{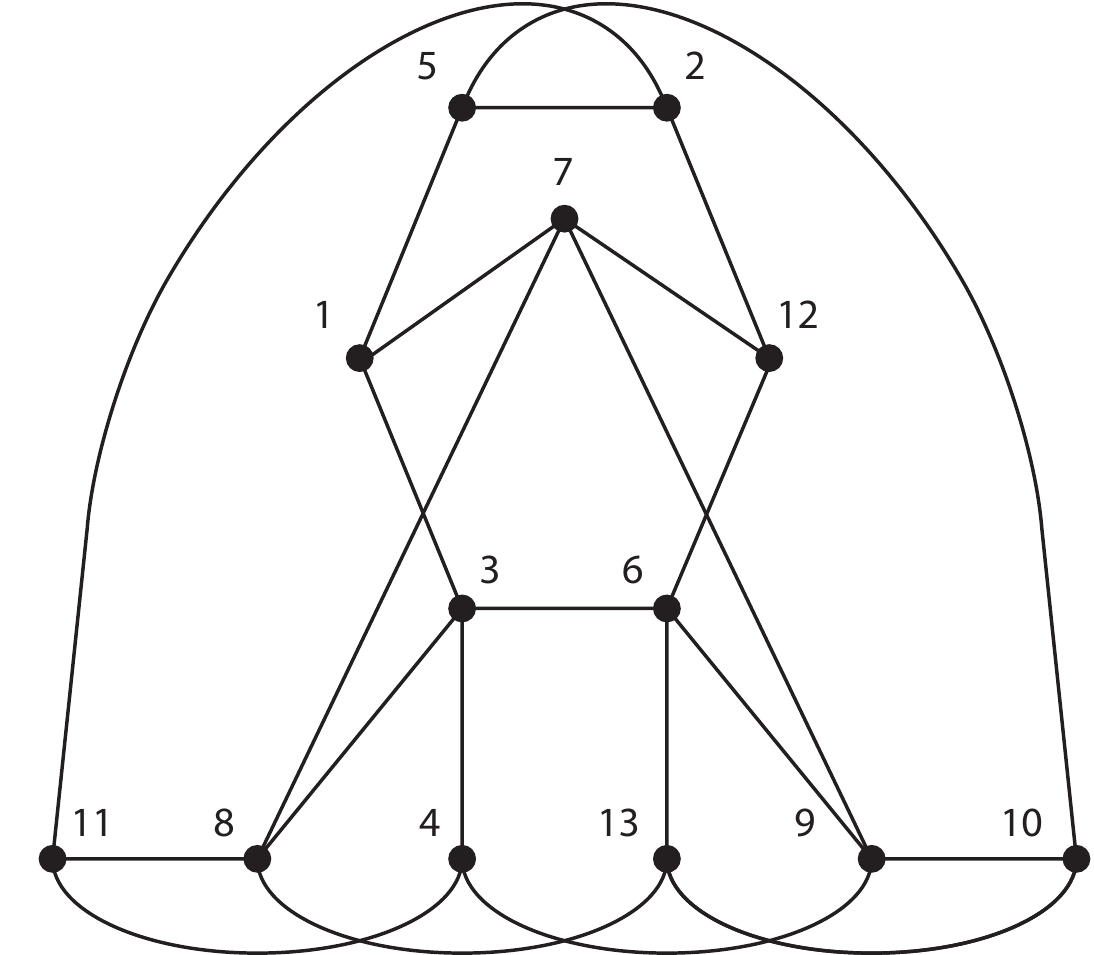}
\caption{Cousin 53 of the $K_{3,3,1,1}$ family.}\label{figC53symm}
\end{center}
\end{figure}

\begin{lemma} \label{lemC53}%
Cousin 53 (Figure~\ref{figC53symm}) of the $K_{3,3,1,1}$ family is MMIK.
\end{lemma}

\begin{proof}
Let $G$ denote Cousin 53. As in the proof above,
all graphs in the family are IK, including $G$.
Since $G$ has no isolated vertices, it
will be enough to show that $G - e$ and $G/e$
have knotless embeddings for every edge $e$ in $G$.

As in Figure~\ref{figC53symm}, the graph has an involution  $(1,12)(2,5)(3,6)(4,13)(8,9)(10,11)$.
This allows us to identify the 22 edges in pairs
with the exception of the edges $(2,5)$ and $(3,6)$
(which are fixed by the involution).
Thus, up to symmetry,
there are 12 choices for the edge $e$ and 24 minors ($G-e$ or $G/e$) to investigate.

\begin{figure}[ht]
\begin{center}
\includegraphics[scale=0.6]{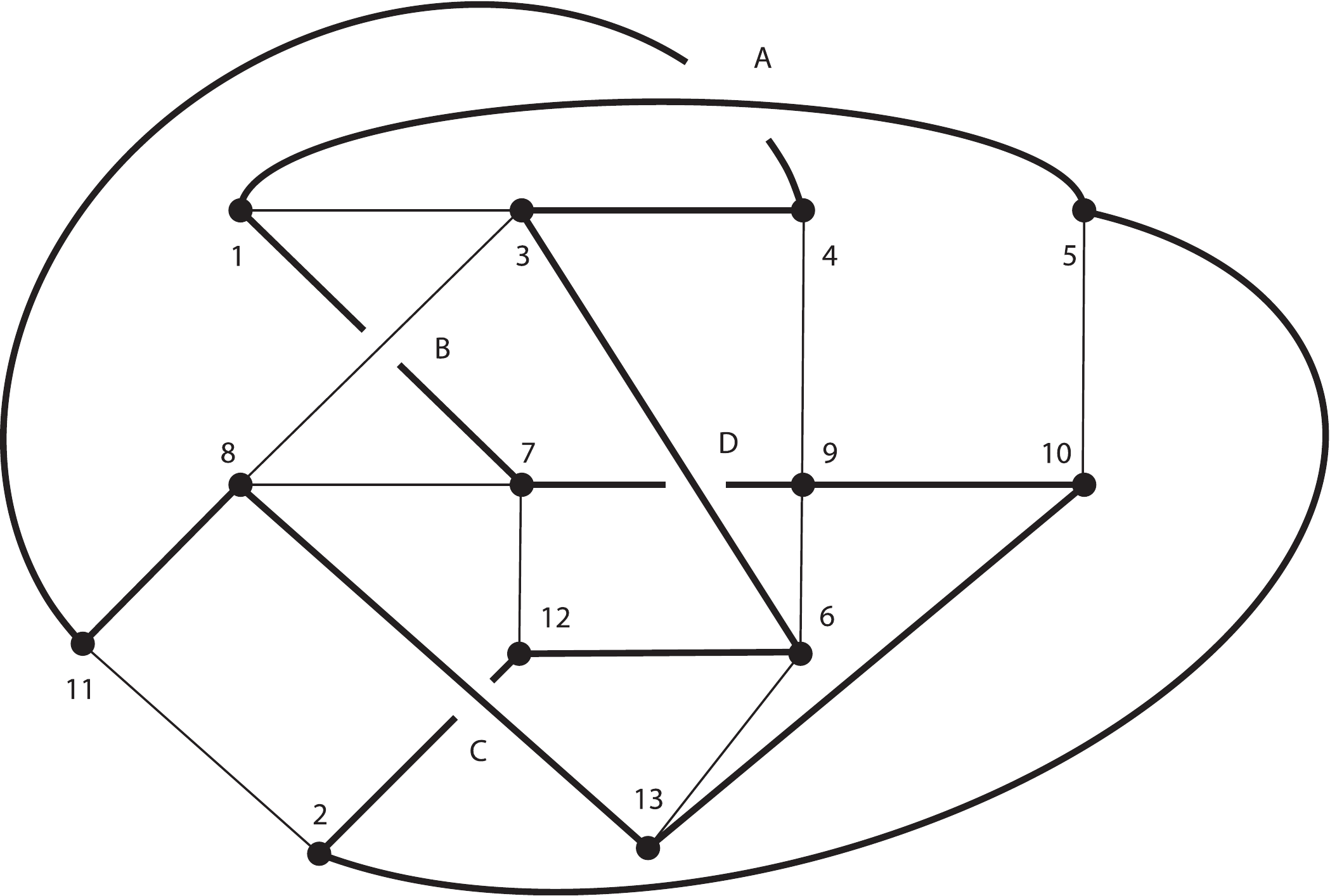}
\caption{An embedding of Cousin 53 which has a unique knotted cycle (in bold).}\label{figC53}
\end{center}
\end{figure}

The argument is based primarily on the embedding of $G$ shown in Figure~\ref{figC53}, for which 
there is a single knotted cycle $(1,5,2,12,6,3,4,11,8,13,10,9,7,1)$,
as well as four crossings, labeled $A$, $B$, $C$, $D$ in the figure. By flipping (i.e., interchanging the over- and undercrossing arcs) at selected crossings, we construct two additional embeddings, each having a unique knotted cycle. Let's call the representation of $G$ shown in the figure Embedding 1.
If we flip the crossing A, we have what we will call Embedding 2
whose unique knotted cycle is $(1, 3, 6, 12, 2, 11, 4, 9, 7, 8, 13, 10, 5, 1)$. For Embedding 3,
we flip the crossings A and C, which gives the knotted cycle
$(1, 3, 6, 13, 8, 11, 4, 9, 7, 12, 2, 5, 1)$.


Out of the 12 choices for an edge $e$, all but one occurs as an
edge either in the  knotted cycle of Embedding 1 or else in that of Embedding 2.
For each such $e$, this gives an unknotted embedding of $G-e$. The
remaining possibility is $e = (6,9)$ (or, equivalently, $(3,8)$).
In this case, deleting vertices $2$ and $3$ from $G-(6,9)$ results in a planar graph
(this is not obvious from Figures~\ref{figC53symm} or \ref{figC53}, but is easy to verify manually or using Mathematica).
Therefore, $G-(6,9)$ is $2$--apex and, by Lemma~\ref{lem2ap}, not IK.
Thus, no minor of the form $G - e$ is IK.

If we contract the edge $e = (1,3)$ in Embedding 1 (shown in Figure~\ref{figC53}), the single knotted cycle becomes two cycles that share the
new vertex formed by identifying vertices $1$ and $3$.  Since
these two cycles are unknots, this is an unknotted
embedding of $G/e$. Similarly, contracting either edge $e = (4,9)$ or
$e = (5,10)$ in Embedding 1 leads to an unknotted embedding of $G/e$.
Embedding 2 shows that $G/e$ is unknotted for $e = (1,7)$, $(2,5)$,
$(6,9)$, and $(8,11)$, while Embedding 3 does for $e = (3,4)$ and $(7,8)$.

For each of the remaining three choices of $e$, we give vertices that,
when deleted from $G/e$, yield a planar graph,
showing that $G/e$ is $2$--apex:
for $e = (1,5)$, delete vertices $4$ and $6$;
for $e = (3,6)$, delete vertices $2$ and $7$;
for $e = (4,11)$, delete vertices $5$ and $6$.
Thus, by Lemma~\ref{lem2ap}, none of these graphs
is IK, completing the argument for the $G/e$ minors.

As no $G-e$ nor $G/e$ minor is IK, we conclude that $G$ is MMIK.
\end{proof}

\section{The $E_9+e$ family}
\label{sectionE9eFamily}
The graph $E_9+e$ (see Figure~\ref{figE9p1}) has nine vertices and 22 edges and is formed by adding the edge $(3,9)$ to $E_9$.
The $E_9+e$ family consists of 110 cousins; due to its large size,
we do not provide here a diagram for the entire family, but only a partial diagram, as explained further below.
Edge lists for all 110 cousins, as well as a diagram of the entire family,
can be found in the Appendix~\cite {GMN}.
The family includes two 8-vertex graphs: the graph whose complement consists of two stars, each on four vertices (see Figure 4vi of \cite{CMOPRW}) and $K_{3,2,1,1,1} - {(b_1,c),(b_1,d)}$, whose complement is a
triangle and a star of four vertices.
We refer the reader to \cite{CMOPRW} for an explanation of this
notation along with a proof that these two graphs are IK (also proved,
independently, in \cite{BBFFHL}). Note that neither is MMIK. The two star graph has $K_7$ as a minor while $K_{3,2,1,1,1} - {(b_1,c),(b_1,d)}$ has $H_8$ (the graph obtained by a $\ty$ move  on $K_7$) as a minor.

\begin{figure}[ht]
\begin{center}
\includegraphics[scale=0.7]{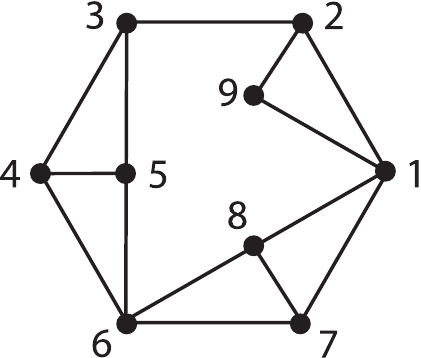}
\caption{The complement of the graph $E_9+e$.}\label{figE9p1}
\end{center}
\end{figure}

The family also
includes three other parentless graphs, all on ten vertices, which we call Cousins 41, 47, and 50. We can describe these graphs by listing their edges:

\noindent%
Cousin 41: $(1, 3), (1, 4), (1, 5), (1, 6), (2, 5), (2, 6), (2, 8), (2, 10), (3, 6), (3, 7), (3, 8)$,\\ $(3, 9), (4, 8), (4, 9), (4, 10), (5, 7), (5, 9), (6, 9), (6, 10), (7, 9), (7, 10), (8, 9)$.

\noindent%
Cousin 47: $(1, 3), (1, 4), (1, 5), (1, 6), (2, 5), (2, 6), (2, 8), (2, 10), (3, 6), (3, 7), (3, 8)$, \\ $(4, 8), (4, 9), (4, 10), (5, 7), (5, 8), (5, 9), (5, 10), (6, 9), (7, 9), (7, 10), (8, 9)$.

\noindent%
Cousin 50: $(1, 3), (1, 4), (1, 5), (1, 6), (1, 10), (2, 5), (2, 6), (2, 8), (2, 10), (3, 6), (3, 7)$, \\ $(3, 8), (3, 9), (4, 8), (4, 9), (4, 10), (5, 7), (5, 8), (5, 9), (6, 9), (7, 9), (7, 10)$.

To show that all graphs in the $E_9+e$ family are IK,
it's enough to check that all the parentless graphs in the family are IK.
We've explained why the two 8-vertex parentless graphs are IK.
The program of \cite{mn} shows that 
the four other parentless graphs are IK.

\begin{figure}[ht]
\begin{center}
\includegraphics[scale=0.7]{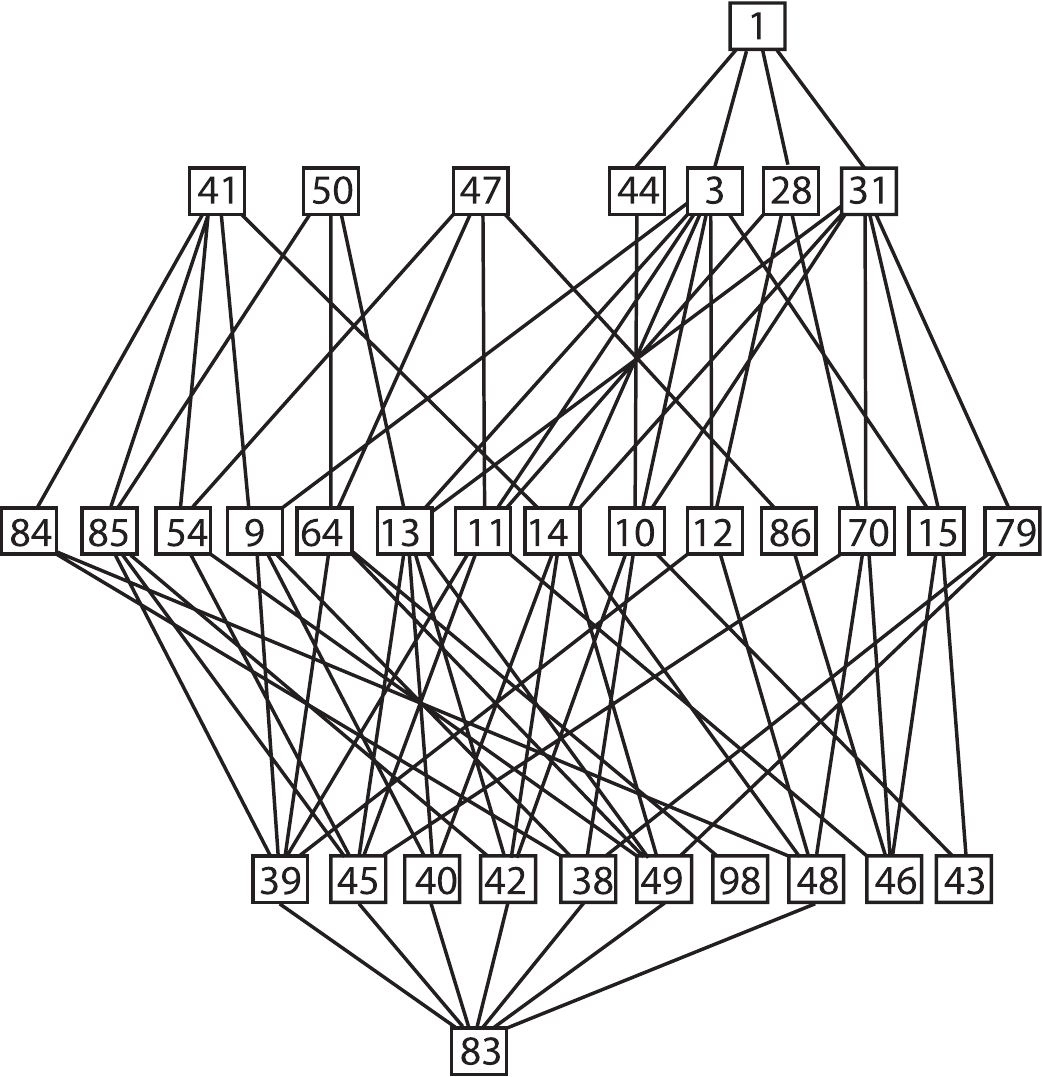}
\caption{The MMIK cousins of $E_9+e$.}\label{figE9p1MM}
\end{center}
\end{figure}

Of the 110 graphs in the family,
only 33 are MMIK;
they are shown in Figure~\ref{figE9p1MM}.
These 33 graphs are the ancestors of Cousins~43, 46, 83, and 98; 
they include $E_9+e$ (Cousin 1)
as well as the three parentless 10-vertex graphs, Cousins 41, 47, and 50. The other graphs in the family are all descendants of the two 8-vertex graphs. As the 8-vertex graphs are not MMIK, it follows, by Lemma~\ref{lemYTMMIK}, that the remaining 77 graphs in the family are not MMIK.

The following lemma shows that Cousin 83 and, hence, its 28 ancestors are MMIK. We omit the similar arguments which show that Cousins 43, 46, and 98 are also MMIK.

\begin{figure}[ht]
\begin{center}
\includegraphics[scale=0.35]{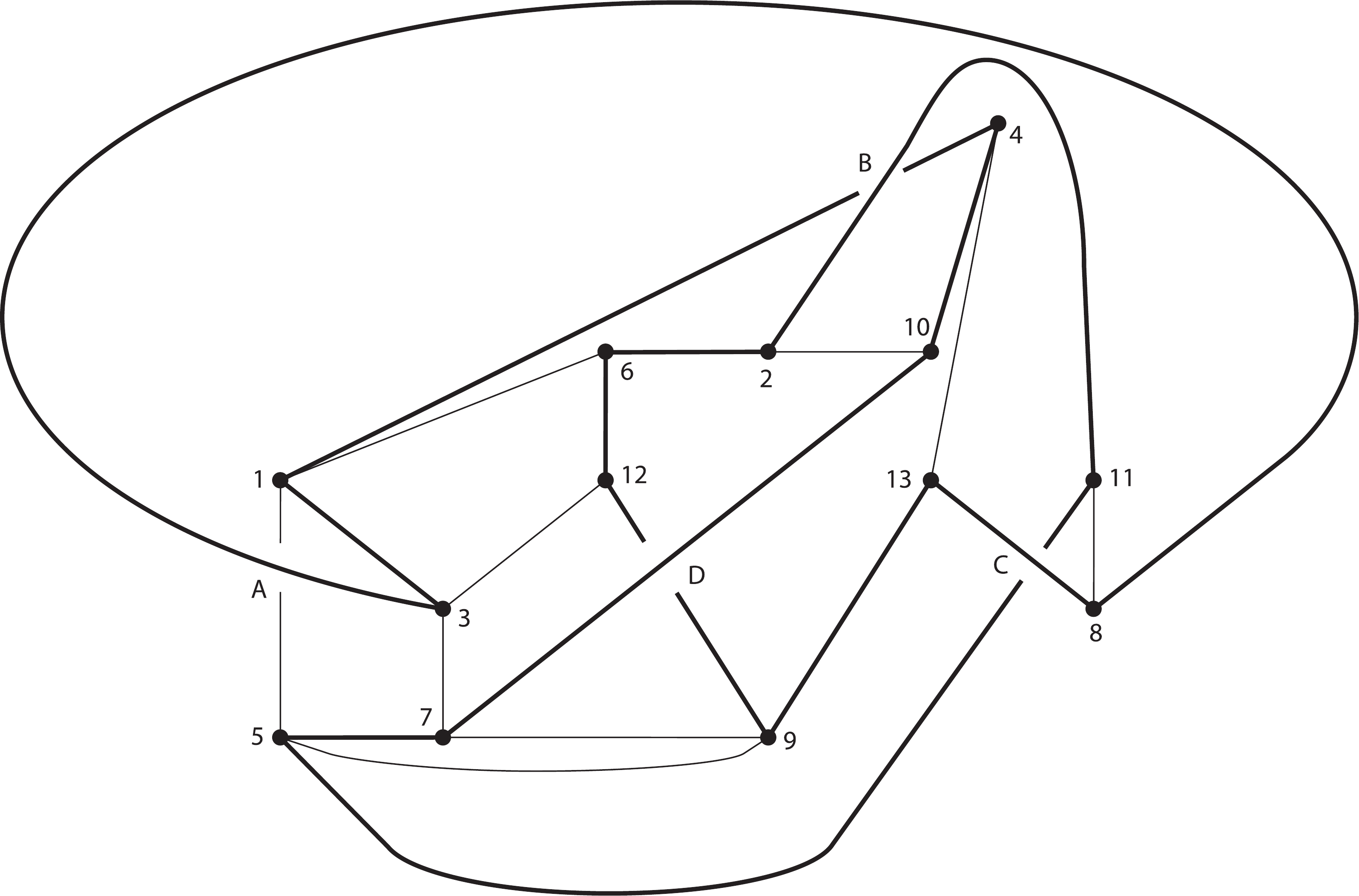}
\caption{An embedding of Cousin 83.}\label{figC83}
\end{center}
\end{figure}

\begin{lemma}
\label{lemma-cousin83}
Cousin 83 (Figure~\ref{figC83}) of the $E_9+e$ family is MMIK.
\end{lemma}

\begin{proof}
The argument is similar to that of Lemma~\ref{lemC53} so we will omit some of the details.
Let $G$ denote Cousin 83.
As $G$ has no symmetries,
the 44 minors obtained by removing or contracting each of the 22 edges
are pairwise non-isomorphic.
We will demonstrate that
none of the 44 graphs $G-e$, $G/e$ are IK.
Figure~\ref{figC83} shows Embedding 1 with its unique knotted cycle $(1,3,8,13,9,12,6,2,11,5,7,10,4,1)$. By flipping crossings we obtain four other embeddings, each with a unique knotted cycle:\\
Embedding 2 (flip B): $(1,4,13,9,12,3,8,11,2,10,7,5,1)$;\\
Embedding 3 (flip A \& B): $(1,5,11,2,10,7,3,8,13,9,12,6,1)$;\\
Embedding 4 (flip C): $(1,4,10,7,3,8,13,9,12,6,2,11,5,1)$; \\
Embedding 5 (flip D): $(1,4,13,8,3,7,10,2,11,5,9,12,6,1)$.

All but one edge $e$ appears in one of the five cycles. The corresponding
embedding shows that $G-e$ is not IK.
For the remaining edge $e = (7,9)$, note that removing vertices 2 and 3
from $G-e$ results in a planar graph.
So, by Lemma~\ref{lem2ap}, $G-(7,9)$ is not IK.
This completes the argument that no minor of the form $G-e$ is IK.

For 13 of the 22 edges,
contracting the edge $e$ turns the unique cycle in at least one
of the five embeddings into two unknotted cycles, showing 
that $G/e$ is not IK. 
For eight of the remaining nine edges, 
(namely $(1,4)$, $(2,11)$, $(3,8)$, $(3,12)$, $(5,11)$, $(7,9)$, $(7,10)$,  and $(9,12)$)
$G/e$ is $2$--apex and, therefore, not IK. Finally, since $G/(6,12)$ is isomorphic to Cousin 19 of the $K_7$ family, it
too is not IK. This completes the argument for minors of the form
$G/e$ and the proof of the lemma.
\end{proof}

\textit{Remark}.
Although it preserves IKness, 
the $\ty$ move  doesn't necessarily preserve MMIKness.
Indeed, Cousin~83, 
which is MMIK by Lemma~\ref{lemma-cousin83},
has a nonMMIK child, Cousin~87.
(As a descendant of both of the nonMMIK 8-vertex graphs in this family, Cousin~87 is also nonMMIK by Lemma~\ref{lemYTMMIK};
an edge list for Cousin~87 can be found in our Appendix~\cite{GMN}).

\section{The $G_{9,28}$ family}

The graph $G_{9,28}$ has nine vertices and 28 edges. It's most easily described in terms of its complement, which is the disjoint union of a $7$--cycle on the vertices $1, 2, \ldots, 7$ and the edge $(8, 9)$.

The $G_{9,28}$ family,  listed fully in \cite{GMN}, consists of 1609 cousins,
 25 of which, including $G_{9,28}$ itself, are parentless.
The remaining cousins are descendants of one or more of these 25 parentless graphs.
We used the computer program of \cite{mn}
to verify that each of these 25 parentless graphs is IK;
hence all 1609 cousins are IK.
Here we give a ``traditional proof'' that $G_{9,28}$ is IK.
We note that $G_{9,28}$ and its descendants account for 1062
of these 1609 cousins;
thus these 1062 graphs are IK even without the ``computer proof.''

We also show in this section that Cousin 1151 of $G_{9,28}$,
which is a descendant of $G_{9,28}$, is MMIK.
Cousin 1151 and its ancestors form a set of 156 graphs
(only 101 of which are $G_{9,28}$ or its descendants).
Thus all of these 156 graphs are MMIK.

\subsection{$G_{9,28}$ is IK}
\label{sectionG28}

\begin{figure}[ht]
\begin{center}
\includegraphics[scale=1]{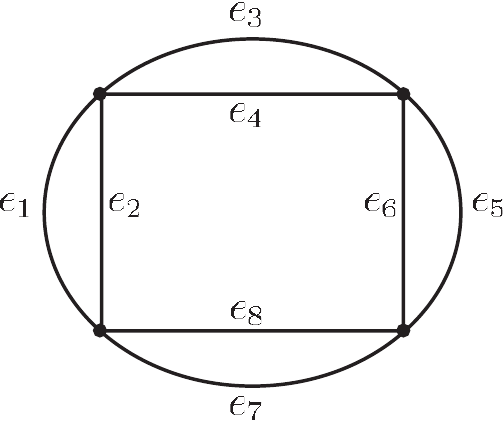}
\caption{The graph $D_4$.}\label{figD4}
\end{center}
\end{figure}

In this subsection, we show that $G_{9,28}$ is IK. We'll make use of
a lemma due independently to Foisy~\cite{F} and Taniyama and Yasuhara~\cite{ty}, which we restate here.
Figure~\ref{figD4} shows the multigraph $D_4$.
For each $i = 1,2,3,4$,
let $C_i$ denote the cycle consisting of the two edges
$e_{2i-1}$ and $e_{2i}$.
For any given embedding of $D_4$,
let $\sigma$ denote the mod~2
sum of the Arf invariants of the $16$ Hamiltonian cycles in that
embedding of $D_4$.
($\Arf(K)$ equals the reduction modulo 2 of the second coefficient of the Conway polynomial of $K$.)
Since the unknot has Arf invariant zero,
if $\sigma \neq 0$ there must
be a nontrivial knot in  the embedding. The lemma shows that this will happen
whenever the mod~2 linking numbers, $\mbox{lk} (C_i, C_j)$, of both
pairs of opposing cycles are non-zero.

\begin{lemma}\cite{F,ty}
\label{D4-lemma}
 Given an embedding of the graph $D_4$, $\sigma \neq 0$  if and only if $\mbox{lk} (C_1, C_3) \neq 0$ and $\mbox{lk} (C_2, C_4) \neq 0$.
\end{lemma}

\begin{figure}[ht]
\begin{center}
\includegraphics[scale=0.7]{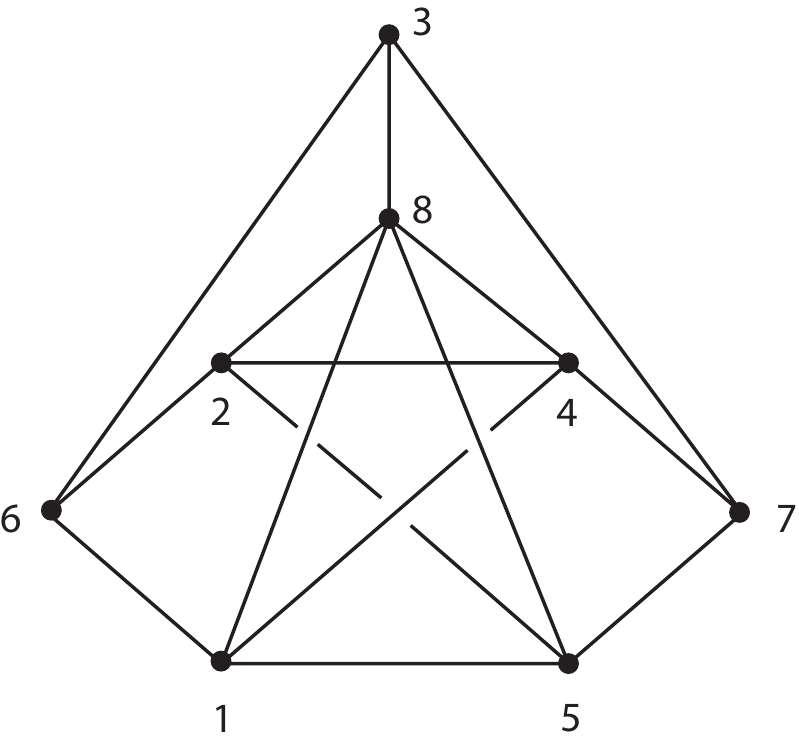}
\caption{The Petersen graph $P_8$ is realized as a subgraph, $S_1$, of $G_{9,28}$.}\label{figS1}
\end{center}
\end{figure}

\begin{prop}
The graph $G_{9,28}$ is IK.
\end{prop}

\begin{proof}{
First observe that the Petersen graph $P_8$
shown in Figure~\ref{figS1}
is a subgraph of $G_{9,28}$
since the complement of $P_8$
contains the 7-cycle $(1,2,3,4,5,6,7)$
and the edge $(8,9)$;
we'll call this subgraph $S_1$.
By cyclically permuting the vertex labels $1, 2, \ldots, 7$ in Figure~\ref{figS1},
we obtain six more subgraphs, $S_2, S_3, \ldots S_7$, of $G_{9,28}$,
each isomorphic to $P_8$.
There are eight pairs of cycles in each $S_i$.
For example, the eight links in $S_1$ and $S_2$ are:

$$ \begin{array}{l||c|c|c|c}
i & l_{i1} & l_{i2} & l_{i3} & l_{i4}  \\
\hline
1 & 148,36257 & 158,36247 & 248,36157 & 258,36147 \\
2 & 258,47361 & 268,47351 & 358,47261 & 368,47251 \\
\multicolumn{5}{c}{} \\
i & l_{i5} & l_{i6} & l_{i7} & l_{i8} \\
\hline
1 & 2475,3618 & 5162,3748 & 1475,3628 & 4261,3758 \\
2 & 3516,4728 & 6273,4158 & 2516,4738 & 5372,4168
\end{array} $$

In the table, we've listed the indices of the vertices in each cycle. Thus $l_{11}$, $S_1$'s first link, consists of the cycles $(1, 4, 8, 1)$ and $(3, 6, 2, 5, 7, 3)$. We will frequently use this abbreviated notation in what follows.
Note that each link $l_{2j}$ can be obtained from the one above it, $l_{1j}$, by applying the cyclic permutation $\gamma = (1,2,3,4,5,6,7)$; we'll write $l_{2j}  = \gamma(l_{1j})$. In a similar way, we
determine the links $l_{ij}$ for each $i = 3,4,5,6,7$ by repeatedly applying $\gamma$.

Fix an arbitrary embedding of $G_{9,28}$. We wish to show that there is a knotted cycle in that embedding.
We'll argue that $G_{9,28}$ has a
$D_4$ minor embedded with opposite cycles linked. 
Using Lemma~\ref {D4-lemma}, this implies there is a knotted cycle in
the $D_4$ and we will refer to such a $D_4$ minor as a ``knotted $D_4$.''
We can then identify the knot in the $D_4$ with a knotted cycle in the given embedding of $G_{9,28}$.

As shown by Sachs~\cite{sa}, 
in any embedding of the Petersen graph $P_8$, the mod 2 sum of the linking number over the eight pairs of cycles is non-zero. 
This means that, in each $S_i$, at least one pair of cycles $l_{ij}$ has non-zero linking number mod 2. To simplify the exposition, for the remainder of this proof, we will use ``linked'' to mean ``has nonzero linking number mod 2.''

\begin{figure}[ht]
\begin{center}
\includegraphics[scale=0.5]{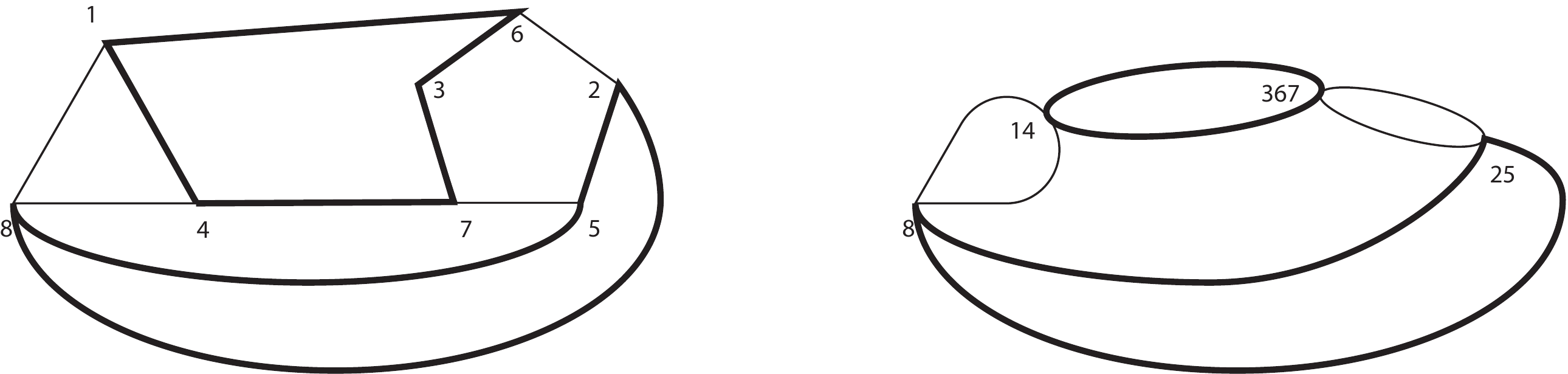}
\caption{Contracting $(1,4)$, $(2, 5)$, $(3,6)$,
and $(3,7)$ results in a $D_4$.}\label{fig1121}
\end{center}
\end{figure}

Suppose first that it's $l_{11}$ that is linked.
We'll use this to deduce that $l_{26}$ or $l_{28}$ is linked.
We'll denote this situation by writing ``$l_{11} \Rightarrow l_{26}$ or $l_{28}$.''
We will argue that any other $l_{2j}$, if linked, would result in a
knotted $D_4$ minor and therefore a knotted cycle in the given  embedding
of $G_{9, 28}$. 
If $l_{21}$ is linked, then contracting the edges $(1,4)$, $(2,5)$, $(3,6)$,
and $(3,7)$ results in a $D_4$ graph with one set of opposing cycles arising  from $l_{11}$ and the other from $l_{21}$ (see Figure~\ref{fig1121}). Assuming both pairs are linked, this results in a nontrivial knot in the $D_4$ graph by Lemma~\ref {D4-lemma} and hence a nontrivial knot in the embedding of $G_{9,28}$. So, we can assume $l_{21}$ is not linked. Similarly, if $l_{27}$ is linked, contracting $(2,5)$, $(2,6)$, $(3,7)$, and $(4,8)$ results in a $D_4$ with linked opposite cycles arising from $l_{11}$ and $l_{27}$.

Suppose that among the $l_{2j}$ pairs,
it's $l_{22}$ that is linked. 
Here, we will first decompose the cycles
of $l_{11}$ and $l_{22}$. Since $(3,5)$ is an edge of $G_{9,28}$, in homology, we can think of the cycle $36257$ (i.e., $(3, 6, 2, 5, 7, 3)$ ) as the sum $[ \gamma_2 + \gamma_3 ]$ of the cycles
$\gamma_2 = 3625$ and $\gamma_3 = 357$.
The sum is linked with the other component of $l_{11}$, $148$, so we
deduce that exactly one of $\gamma_2$ and $\gamma_3$ is also linked with
$148$ (see \cite[Lemma 3.1]{F}). We will refer to this way
of dividing $36257$ into $3625$ and $357$ as
``cutting along $35$.'' If it's $3625$ that's linked with $148$, then
contracting $(1,4)$, $(3,5)$, and $(2,6)$ results in a $D_4$ whose
opposite cycles arise from the linked cycles $3625$ and $148$ and the linked cycles of $l_{22}$. Consequently, by Lemma~\ref{D4-lemma}, the embedding of $G_{9,28}$
has a knotted cycle in this case. If instead it's $357$ that's linked with
$148$, we will need to cut the cycle $47351$ of $l_{22}$ along $13$.
This leaves two cases. If it's $4731$ that links the other component of
$l_{22}$, $268$, then after contracting $(1,4)$, $(2,5)$, $(2,6)$, and
$(3,7)$, we'll have a knotted embedding of $D_4$.
On the other hand, if it's $135$ that links $268$, we'll want to contract
$(1,4)$, $(2,6)$, $(2,7)$ and $(3,5)$ to achieve an embedding of $D_4$ that implies a
knotted cycle in $G_{9,28}$ by Lemma~\ref{D4-lemma}.

Similarly, if $l_{23}$ or $l_{25}$ is linked, we'll need to cut $36257$ along $35$. For $l_{24}$, we again cut $36257$ along $35$ and
further, cut $47251$ along $42$. 
Thus, in every case other than $j=6$ and $j=8$, we've shown
that assuming $l_{11}$ and $l_{2j}$ are both linked leads to an embedding
of $D_4$ that forces a knotted cycle in our embedding of
$G_{9,28}$. This shows that $l_{11} \Rightarrow l_{26}$ or $l_{28}$.

In much the same way, we now show $l_{11} \Rightarrow l_{53}$ or $l_{56}$. It's straightforward to verify that there'll be a knotted $D_4$ in case 
both $l_{11}$ and one of $l_{51}$, $l_{52}$, or $l_{57}$ are linked. 
As for $l_{54}$, it's the same link as $l_{22}$, which we treated above.  The two remaining cases require cutting, as we will now describe. If $l_{11}$ and $l_{55}$ are 
both linked, cut $7358$ along $57$ and use the edges $(6,9)$ and
$(7,9)$ to identify $36257$ as the sum (in homology) of $3697$ and $62579$.
(We'll call this operation ``cutting along 697.'') 
Finally, if it's $l_{58}$ that is linked, cut $7428$ along $27$ and 
$36257$ along $35$.

As a final step in the argument for $l_{11}$, we construct a new $P_8$
subgraph, $T_1$, from $S_1$ by
interchanging the vertex labels $8$ and $9$. Thus, the linked pairs in $T_1$ are
$m_{11} = 149,36257$; $m_{12} = 159,36247$; $\ldots$;
$m_{18} = 4261,3759$ (compare with the table of $l_{1j}$ above). 
Again, by \cite{sa}, at least one of these $m_{1j}$ is linked in
any embedding of $G_{9,28}$.
We'll argue that this, 
together with
$l_{11} \Rightarrow l_{26}$ or $l_{28}$
and 
with $l_{11} \Rightarrow l_{53}$ or $l_{56}$,
imply that
there is a knotted $D_4$ in our embedding of $G_{9,28}$.

First notice that $l_{11}$ will form a $D_4$ with $m_{12}$, $m_{13}$, $m_{15}$, $m_{16}$, and
$m_{18}$ (after cutting $36257$ along $27$). 
In other words, $l_{11} \Rightarrow m_{11}$, $m_{14}$, or $m_{17}$.
Now, each of the following pairs of link forms a $D_4$:
$m_{11}$ and $l_{26}$ (cut 36257 along 35),
$m_{14}$ and $l_{26}$ (no cuts),
$m_{14}$ and $l_{28}$ (cut 36147 along 13),
and
$m_{11}$ and $l_{28}$ (cut 36257 along 85, 86, 87, i.e., $36257 = 3687 + 785 + 5862$).
Since $l_{11} \Rightarrow l_{26}$ or $l_{28}$, we deduce that if $l_{11}$ is linked, then we can assume $m_{17}$ is linked, too, i.e., $l_{11} \Rightarrow m_{17}$.

To complete the argument for $l_{11}$, 
we construct $T_2$, another embedding of $P_8$,
and its links $m_{2j}$, in the usual way by applying the  
$7$--cycle $\gamma$ to the vertex indices of $T_1$.
Using the same type of argument as above,
we can see that
$m_{17} \Rightarrow m_{22}$ or $m_{24}$
(the only case that requires any cuts
is when $m_{17}$ is combined with $m_{28}$,
where we cut
$3629$ along $382$
and  $5372$ along $57$).
As we showed earlier, 
$l_{11} \Rightarrow l_{26}$ or $l_{28}$; 
and we can show $m_{22}$ combined with
either $l_{26}$ or $l_{28}$ yields a $D_4$
(for $m_{22}$ with $l_{26}$, cut 47351 along 13;
for $m_{22}$ with $l_{28}$, no cuts);
we therefore conclude that if $m_{22}$ and $l_{11}$ are both linked, then there is a knotted $D_4$. 
Similarly, we know $l_{11} \Rightarrow l_{53}$ or $l_{56}$; 
and we can show $m_{24}$ combined with either $l_{53}$ or $l_{56}$ produces a $D_4$
(for $m_{24}$ with $l_{53}$, cut 47251 along 187;
for $m_{24}$ with $l_{56}$, cut 47251 along 24);
hence $l_{11}$ and $m_{24}$ give a $D_4$. 
It follows that if $l_{11}$ and $m_{17}$ are both linked,
there is a nontrivial knot in our embedding of $G_{9,28}$. 
Thus, when $l_{11}$ is linked, 
no matter which pair of cycles $m_{1j}$, $j = 1, \ldots, 8$ is linked in the Petersen
graph $T_1$, we will have a knotted cycle in
our embedding of $G_{9,28}$. 
This completes the argument for $l_{11}$.

Next, we show that for all $i = 1, \ldots, 7$,
we can assume all $l_{ij}$ except $l_{i3}$, $l_{i5}$, and $l_{i8}$ are unlinked.
Indeed, we have shown that if $l_{11}$ is linked,
there will be a knot in our embedding of $G_{9,28}$, which is the goal
of our proof. Thus, 
we can assume $l_{11}$ is not linked. By symmetry, the same argument can be applied to each $l_{i1}$, $i = 1, \ldots, 7$. 
Since $l_{14}$ becomes $l_{11}$ after applying the
involution $\delta = (1,5)(2,4)(6,7)$, which is a symmetry of $G_{9,28}$, 
we can likewise assume $l_{14}$, and hence
every $l_{i4}$ is unlinked. 
Also, $l_{12} = l_{51}$ which, as we have already noted, is not linked. So, we may assume, no $l_{i2}$ is linked.
Next, suppose $l_{17}$ is linked.
We have mentioned that $m_{17} \Rightarrow m_{22}$ or
$m_{24}$, and, by symmetry, the same argument shows $l_{17} \Rightarrow l_{22}$ or $l_{24}$. 
However, as we have noted, 
no $l_{i2}$ or $l_{i4}$ is linked,
or else we will have a knotted $D_4$.
Thus, $l_{17}$ is not linked either. 
Again, by symmetry, this implies no $l_{i7}$ is linked. Since $l_{16} = \delta (l_{17})$, we can also assume $l_{16}$, and hence all $l_{i6}$, are also not linked.
Thus, only pairs of the form $l_{ij}$ with $j = 3$, $5$, or $8$ are linked.

Now, if $l_{13}$ and $l_{23}$ are both linked,
we get a knotted $D_4$ by doing a few cuts at various stages
(cut $36157$ along $35$,
$47261$ along $42$,
$4261$ along $496$,
and $3615$ along $391$).
Thus, $l_{13} \Rightarrow l_{25}$ or $l_{28}$;
and, by symmetry, we have

\begin{equation}
\label{eqn-3to5}
l_{i3} \Rightarrow l_{(i+1)5} \mbox{ or }l_{(i+1)8}
\end{equation}
where, 
whenever the first index $i$ of $l_{ij}$ is greater than 7 or less than 1 (which comes up further below),
we reduce $i$ mod 7, except that we use 7 instead of 0.
We also argue that $l_{15} \Rightarrow l_{53}$, 
by showing $l_{15}$ and $l_{55}$ together give a $D_4$ (cut 7358 along 57)
and $l_{15}$ with $l_{58}$ together give a $D_4$ (no cuts).
It follows from symmetry that 
\begin{equation}
\label{eqn-5to3}
l_{i5} \Rightarrow l_{(i+4)3}
\end{equation}

We now apply the permutation $\delta=(1,5)(2,4)(6,7)$ to implication~(\ref{eqn-5to3}) above.
First, recall that $\gamma^k (l_{ij}) = l_{(i+k)j}$,
where, as before,
$\gamma$ is the 7-cycle $(1,2,3,4,5,6,7)$.
Also, note that $\delta \gamma = \gamma^{-1} \delta$.
Hence, since $\delta (l_{15}) = l_{18}$,
we get
 $\delta (l_{i5}) = \delta \gamma^{i-1} (l_{15}) = (\gamma^{-1})^{i-1} \delta (l_{15})
 =  \gamma^{1-i}  (l_{18}) = l_{(2-i)8}$.
 Also, $\delta(l_{13})= l_{13}$,
 which, by a similar argument as above, gives
 $\delta(l_{i3}) = l_{(2-i)3}$.
Thus, applying $\delta$ to implication~(\ref{eqn-5to3}) gives
$l_{(2-i)8} \Rightarrow l_{(2-i-4)3}$,
which, by replacing both occurrences of $2-i$ with $i$, gives
\begin{equation}
\label{eqn-8to3}
l_{i8} \Rightarrow l_{(i+3)3}
\end{equation}

Combining implications (\ref{eqn-3to5}), (\ref{eqn-5to3}), and (\ref{eqn-8to3})
yields
$l_{i3} \Rightarrow l_{(i+4)3} \mbox{ or } l_{(i+5)3}$.
In particular, 
$l_{13} \Rightarrow l_{53} \mbox{ or } l_{63}$,
$l_{53} \Rightarrow l_{23} \mbox{ or } l_{33}$,
and
$l_{63} \Rightarrow l_{33} \mbox{ or } l_{43}$.
These three implications together give
$l_{13} \Rightarrow l_{23} \mbox{ or } l_{33} \mbox{ or } l_{43}$.
We've already seen that $l_{13}$ and $l_{23}$ together give a $D_4$.
We also check that
$l_{13}$ and $l_{43}$ give a $D_4$ 
(cut 36157 along 697, 61579 along 59, 62413 along 64, 6413 along 61).
Thus we conclude that
$l_{13} \Rightarrow l_{33}$,
which, by symmetry, gives
\begin{equation}
\label{eqn-33}
l_{i3} \Rightarrow l_{(i+2)3}
\end{equation}
Applying implication (\ref {eqn-33}) repeatedly gives
$l_{13} \Rightarrow l_{33} \Rightarrow l_{53} \Rightarrow l_{73} \Rightarrow l_{23}$,
which means if $l_{13}$ is linked, we get a knotted $D_4$
(since  $l_{13}$ and $l_{23}$ give a $D_4$).
By symmetry, we get a knotted $D_4$ if any $l_{i3}$ is linked.
This, and implications  (\ref{eqn-5to3}) and (\ref{eqn-8to3}),
together imply that if any $l_{i5}$ or $l_{i8}$ is linked,
we get a knotted $D_4$.

This completes the proof that $G_{9,28}$ is IK. No matter
which $l_{1i}$ is linked, we have found
a nontrivial knot in the given embedding of $G_{9,28}$.

}\end{proof}

\subsection{Cousin 1151 of $G_{9,28}$}
\label{Sec-C1151}

We now focus on the childless Cousin 1151 of $G_{9,28}$
(Figure~\ref{figG28C1151abstract})
and show that it is minor minimal.
This implies, by Lemma~\ref{lemYTMMIK}, that the 156 graphs consisting of Cousin 1151 and its ancestors are all MMIK.

\begin{figure}[ht]
\begin{center}
\includegraphics[scale=0.5]{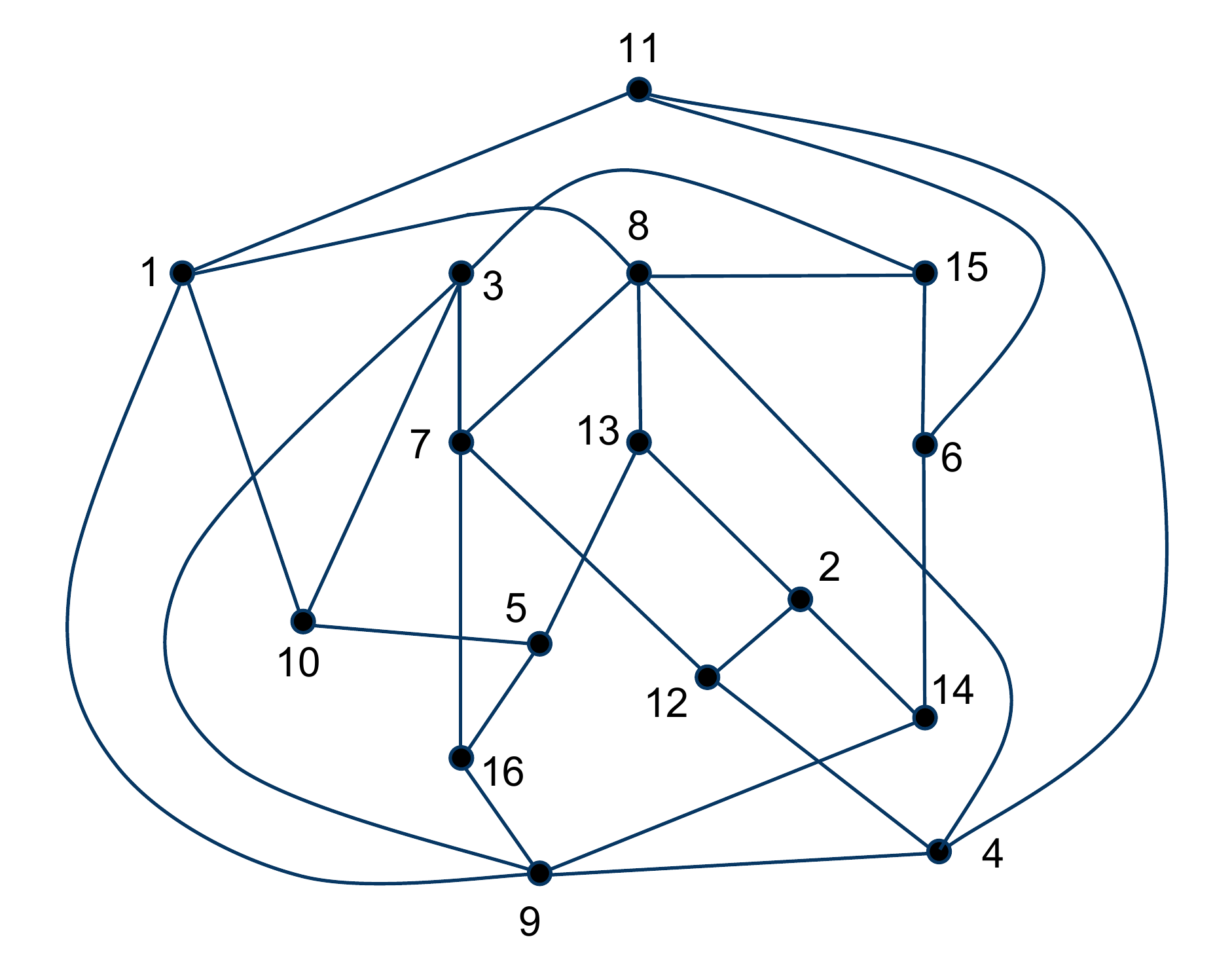}
\caption{Cousin 1151 of $G_{9,28}$.}
\label{figG28C1151abstract}
\end{center}
\end{figure}


It turns out that Cousin 1151 has no symmetries;
hence we consider all its 56 minors obtained by
deleting or contracting each of its 28 edges.
Of these 56 minors, 54 are $2$--apex.
Below, we list each of these 54 graphs as
$G-e$ or $G/e$,
followed by the two vertices that can be removed to obtain a planar graph.
The remaining two graphs are listed as ``not $2$--apex''.
Note that whenever we contract an edge $(a,b)$ in $G$,
we relabel some of the vertices in $G/(a,b)$, as follows:
If $a <b$, then we use the label $a$ for the vertex that edge $(a,b)$ contracts to;
furthermore, we take the vertex in $G$ with the largest label 
and relabel it as vertex $b$ in $G/(a,b)$.

\noindent
$G-(1, 8)$, $\{4, 5\}$; $G/(1, 8)$, $\{1, 2\}$;
$G-(1, 9)$, $\{4, 14\}$; $G/(1, 9)$, $\{1,2\}$;
$G-(1, 10)$, $\{2, 8\}$; $G/(1, 10)$, $\{4, 14\}$;
$G-(1, 11)$, $\{2, 3\}$; $G/(1,11)$, $\{5, 14\}$;
$G-(2, 12)$, $\{1, 3\}$; $G/(2, 12)$, $\{2, 8\}$;
$G-(2, 13)$, $\{1,7\}$; $G/(2, 13)$, $\{1, 2\}$;
$G-(2, 14)$, $\{1, 5\}$; $G/(2, 14)$, $\{2, 3\}$;
$G-(3,7)$, $\{1, 2\}$; $G/(3, 7)$, $\{3, 4\}$;
$G-(3, 9)$, $\{6, 7\}$; $G/(3, 9)$, $\{1, 2\}$;
$G-(3, 10)$, $\{2, 4\}$; $G/(3, 10)$, $\{3, 8\}$;
$G-(3, 15)$, $\{1, 9\}$; $G/(3,15)$, $\{2, 3\}$;
$G-(4, 8)$, $\{7, 13\}$; $G/(4, 8)$, $\{2, 3\}$;
$G-(4, 9)$, $\{1,7\}$; $G/(4, 9)$, $\{2, 3\}$;
$G-(4, 11)$, $\{2, 3\}$; $G/(4, 11)$, $\{7, 13\}$;
$G-(4,12)$, $\{2, 8\}$; $G/(4, 12)$, $\{1, 13\}$;
$G-(5, 10)$, $\{3, 8\}$; $G/(5, 10)$, $\mbox{not $2$--apex}$;
$G-(5, 13)$, $\{1, 2\}$; $G/(5, 13)$, $\{5, 6\}$;
$G-(5, 16)$, $\{1,2\}$; $G/(5, 16)$, $\mbox{not $2$--apex}$;
$G-(6, 11)$, $\{5, 7\}$; $G/(6, 11)$, $\{2, 3\}$;
$G-(6, 14)$, $\{1,7\}$; $G/(6, 14)$, $\{5, 6\}$;
$G-(6, 15)$, $\{1, 7\}$; $G/(6, 15)$, $\{6, 7\}$;
$G-(7,8)$, $\{3, 4\}$; $G/(7, 8)$, $\{1, 2\}$;
$G-(7, 12)$, $\{1, 9\}$; $G/(7, 12)$, $\{3,14\}$;
$G-(7, 16)$, $\{3, 4\}$; $G/(7, 16)$, $\{1, 2\}$;
$G-(8, 13)$, $\{1, 2\}$; $G/(8,13)$, $\{3, 4\}$;
$G-(8, 15)$, $\{2, 3\}$; $G/(8, 15)$, $\{1, 9\}$;
$G-(9, 14)$, $\{2,3\}$; $G/(9, 14)$, $\{1, 7\}$;
$G-(9, 16)$, $\{1, 2\}$; $G/(9, 16)$, $\{6, 13\}$.

By Lemma~\ref {lem2ap},
all the 54 graphs that are  $2$--apex have knotless embeddings.
In Figures~\ref {figG28C1151contract-5-10} and \ref {figG28C1151contract-5-16}
we display knotless embeddings for
the two graphs that are not $2$--apex.
(We used a computer program to verify that every cycle in these two embeddings is a trivial knot.)

\begin{figure}[ht]
\begin{center}
\includegraphics[scale=0.45]{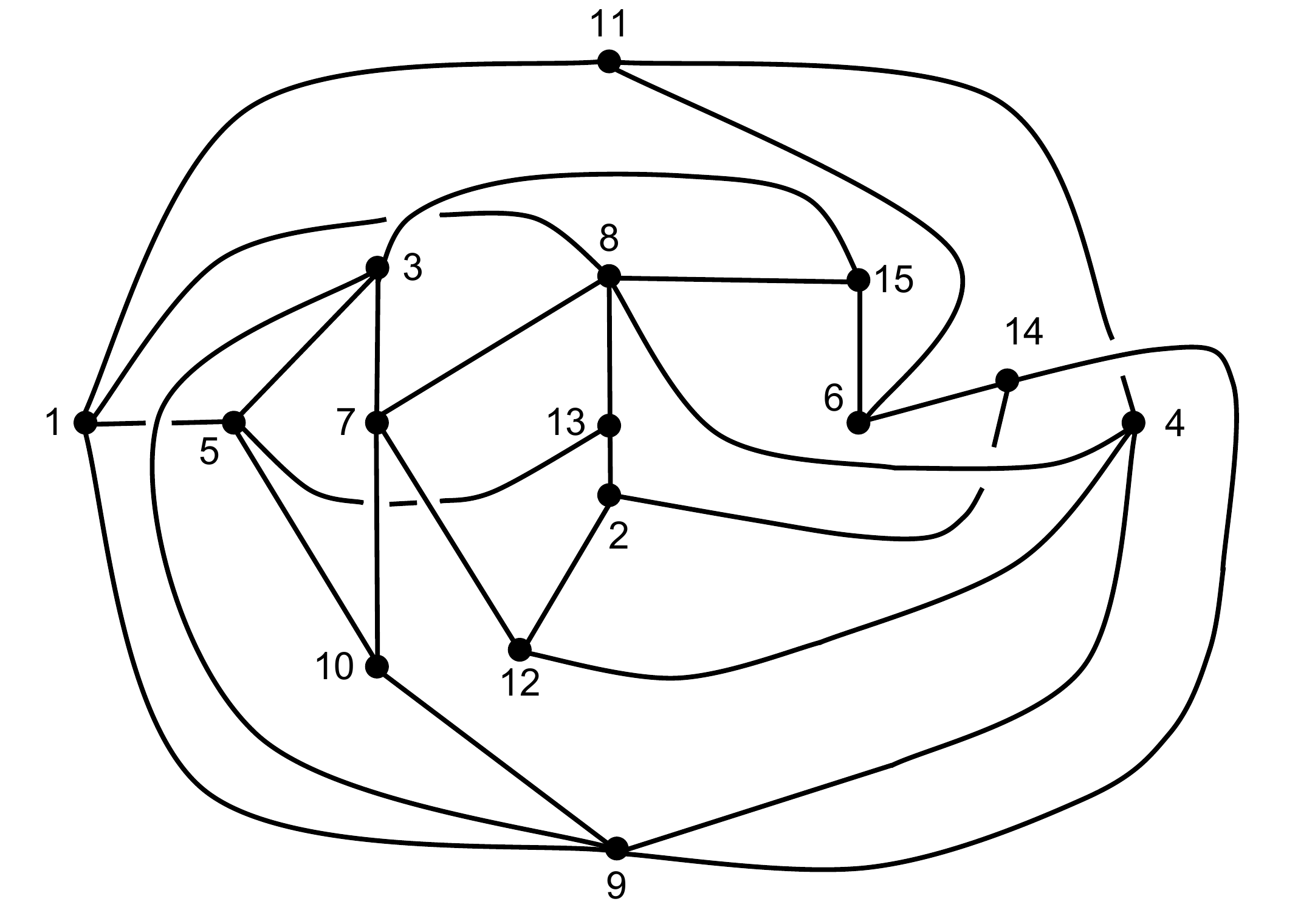}
\caption{A knotless embedding of the graph obtained by contracting edge $(5,10)$
in Cousin 1151 of $G_{9,28}$ (vertex 10 used to be vertex 16).}
\label {figG28C1151contract-5-10}
\end{center}
\end{figure}

\begin{figure}[ht]
\begin{center}
\includegraphics[scale=0.5]{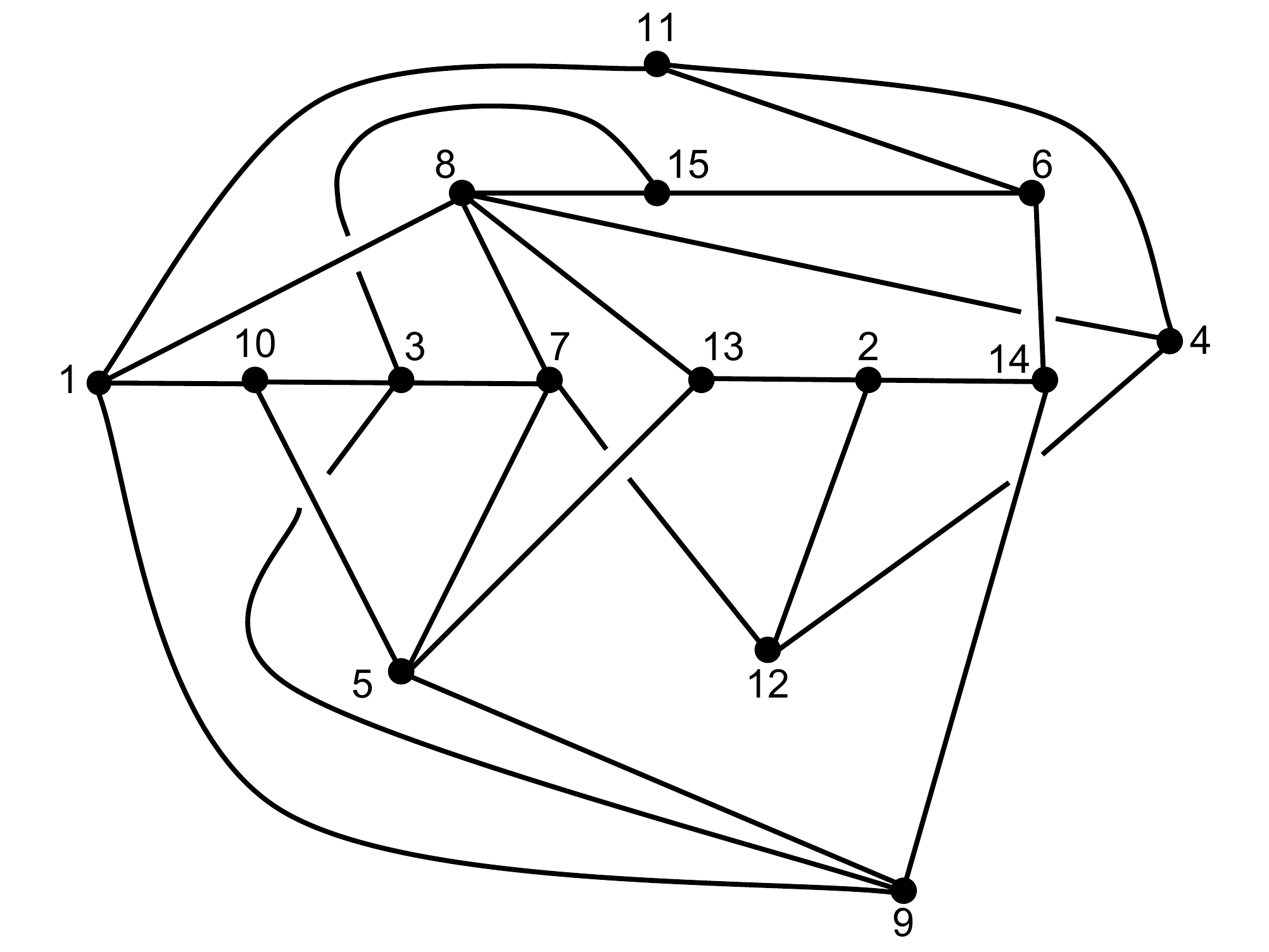}
\caption {A knotless embedding of the graph obtained by contracting edge $(5,16)$
in Cousin 1151 of $G_{9,28}$.}
\label {figG28C1151contract-5-16}
\end{center}
\end{figure}

\section{The $G_{14,25}$ Family}

The graph $G_{14,25}$,
depicted in Figure~\ref{figG1425},
 has 14 vertices and 25 edges:
$(1, 6)$, $(1, 9)$,
$(1, 10)$, $(1, 11)$, $(2, 6)$, $(2, 7)$, $(2, 8)$,
$(2, 14)$, $(3, 10)$, $(3, 12)$, $(3, 13)$, $(4, 6)$, $(4, 7)$, $(4, 9)$, $(4, 11)$,
$(5, 7)$, $(5, 8)$, $(5, 10)$, $(5, 14)$, $(6, 13)$, $(7, 12)$, $(8, 11)$,
$(8, 13)$, $(9, 12)$, $(9, 14)$.

\begin{figure}[ht]
\begin{center}
\includegraphics[scale=1.0]{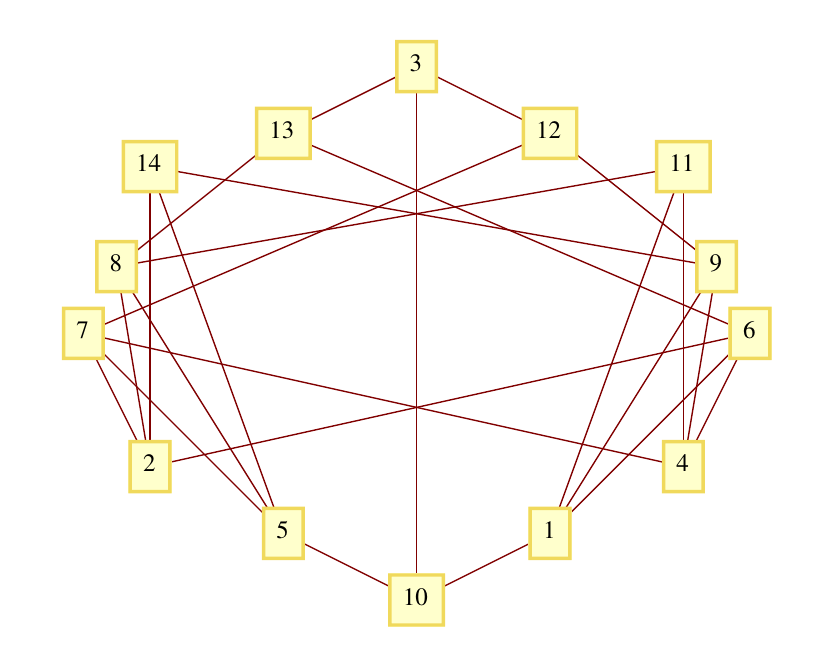}
\caption {The graph $G_{14,25}$.}
\label {figG1425}
\end{center}
\end{figure}

We obtained this graph
by starting with one of the cousins of $G_{9,28}$
that is IK but not MM,
and repeatedly deleting or contracting edges
(a total of three edges)
until we arrived at a MMIK graph.

The graph $G_{14,25}$ is interesting
since it is a MMIK graph with over 600,000 cousins!
We don't know exactly how many cousins it has;
we stopped the computer program after about one week of continuous operation,
since we had no upper bound on the number of cousins
and therefore had no idea how much longer the program might continue to run.
We sampled a small number of these cousins,
which turned out not to be MMIK.
Nevertheless, we wouldn't be surprised if such a  large family turned out to contain
hundreds or thousands of MMIK graphs.

\begin{lemma}
The graph $G_{14,25}$  is MMIK.
\end{lemma}

\begin{proof}
Let $G$ denote the graph $G_{14,25}$.
We show that $G$ is IK by using the computer program described in \cite{mn} to verify that there is a $D_4$ minor with a knotted Hamiltonian cycle in every embedding of the graph.
To prove that $G$ is MM,
since it has no isolated vertices,
it will be enough to show that
for every edge $e$ in $G$,
neither $G - e$ nor $G/e$ is IK.

The graph $G$ has an involution
$(1,5)(2,4)(6,7)(8,9)(11,14)(12,13)$
which allows us to identify all its 25 edges in pairs,
with the exception of the edge $(3,10)$
(which is fixed by the involution).
Thus, up to symmetry, there are 13 choices
for the edge $e$ and 26 minors ($G-e$ or $G/e$) to investigate.
Each of these 26 minors turns out to be $2$--apex.
Below, we list each of them as
$G-e$ or $G/e$,
followed by the two vertices that can be removed to obtain a planar graph.
(See the note in Section~\ref{Sec-C1151} about vertex relabeling.)

\noindent
$G-(1, 6)$, $\{2, 3\}$;  $G/(1, 6)$, $\{1, 6\}$;
$G-(1, 9)$, $\{2, 4\}$;  $G/(1, 9)$, $\{1, 2\}$;
$G-(1, 10)$, $\{2, 4\}$;  $G/(1, 10)$, $\{2, 4\}$;
$G-(1, 11)$, $\{2, 3\}$;  $G/(1, 11)$, $\{1, 7\}$;
$G-(2,6)$, $\{1, 7\}$;  $G/(2, 6)$, $\{2, 3\}$;
$G-(2, 7)$, $\{1, 3\}$;  $G/(2, 7)$, $\{1, 2\}$;
$G-(2, 8)$, $\{3, 5\}$;  $G/(2, 8)$, $\{1, 3\}$;
$G-(2, 14)$, $\{1, 3\}$;  $G/(2, 14)$, $\{1, 2\}$;
$G-(3,10)$, $\{2, 4\}$;  $G/(3, 10)$, $\{2, 4\}$;
$G-(3, 12)$, $\{1, 2\}$;  $G/(3, 12)$, $\{3, 4\}$;
$G-(6,13)$, $\{1, 7\}$;  $G/(6, 13)$, $\{2, 5\}$;
$G-(8, 11)$, $\{1, 7\}$;  $G/(8, 11)$, $\{2, 3\}$;
$G-(8, 13)$, $\{2, 3\}$;  $G/(8, 13)$, $\{1, 7\}$.

It follows from Lemma~\ref {lem2ap}
that each of these 26 minors has a knotless embedding,
and hence $G_{14, 25}$ is MMIK.
\end{proof}

\section*{Acknowldedgements}
The research was supported in part by the Undergraduate Research Center at Occidental College
and by DMS-0905300 at Occidental College.

%
%
%
%

\end{document}